\documentclass[11pt,leqno]{amsart}

\usepackage[utf8x]{inputenc}
\usepackage[T1]{fontenc}
\usepackage[english]{babel}
\usepackage{enumitem}

\usepackage[dvipsnames]{xcolor}
\usepackage{amsthm}
\usepackage[pagebackref, colorlinks=true, linkcolor=MidnightBlue, citecolor=OliveGreen]{hyperref}
\usepackage[capitalise,nameinlink]{cleveref}
\usepackage{amsmath,amssymb}
\usepackage{mathdots}

\usepackage[symbol]{footmisc}

\usepackage{tikz}
\usepackage{tikz-cd}

\usepackage{thmtools}
\usepackage{thm-restate}

\newtheorem*{thm}{Main Theorem}
\newtheorem*{que}{Main Question}

\newtheorem{theorem}{Theorem}
\newtheorem{lemma}[theorem]{Lemma}
\newtheorem{proposition}[theorem]{Proposition}
\newtheorem{corollary}[theorem]{Corollary}

\theoremstyle{definition}

\newtheorem{remark}[theorem]{Remark}
\newtheorem{question}[theorem]{Question}

\newcommand{\numth}{\textsuperscript{th}~}

\DeclareMathOperator{\N}{\mathbb{N}}

\DeclareMathOperator{\Norm}{N}
\DeclareMathOperator{\id}{id}

\DeclareMathOperator{\Sym}{Sym}

\DeclareMathOperator{\Max}{Max^{\not\triangleleft}}
\DeclareMathOperator{\ins}{ins}

\DeclareMathOperator{\Aut}{Aut}
\DeclareMathOperator{\Fin}{Fin}
\DeclareMathOperator{\St}{St}
\DeclareMathOperator{\st}{st}
\DeclareMathOperator{\Rist}{Rist}
\DeclareMathOperator{\rist}{rist}

\renewcommand*{\backref}[1]{}
\renewcommand*{\backrefalt}[4]{%
  \ifcase #1 %
    No citations.
  \or
    (Cited on page~#2.)%
  \else
    (Cited on pages~#2.)%
  \fi%
}

\begin{document}

\makeatletter
\@namedef{subjclassname@2020}{\textup{2020} Mathematics Subject Classification}
\makeatother

\title{Branch groups with many maximal subgroups}

\author[J.\,M. Petschick]{J. Moritz Petschick}
\address{
Fakult\"at f\"ur Mathematik,
Universit\"at Bielefeld,
33501 Bielefeld, Germany}
\email{jpetschick@math.uni-bielefeld.de}



\thanks{The author was supported by the Deutsche Forschungsgemeinschaft (DFG, German Research Foundation) – Project-ID~491392403 – TRR~358.}

\date{\today}

\begin{abstract}
	This article presents the construction of finitely generated branch groups with uncountably many maximal subgroups using embedding techniques. This addresses a question posed by Grigorchuk.
\end{abstract}

\maketitle

\vspace{-1em}
\section{Introduction} 
\label{sec:introduction}

Since Pervova discovered in~\cite{Per05} that certain important branch groups do not admit maximal subgroups of infinite index, much work has been done to understand the maximal subgroups of general finitely generated branch groups. On the one hand, Pervova’s result was generalised to broader classes of groups by various authors, see~\cite{Fra20, FT22, KT18}. On the other hand, Bondarenko demonstrated (non-constructively) in~\cite{Bon10} that branch groups may contain maximal subgroups of infinite index. Francoeur and Garrido in~\cite{FG18} provided a concrete instance of a maximal subgroup of infinite index in a branch group, and, furthermore, proved that non-periodic {\v S}uni{\'k} groups admit no more than countably many maximal subgroups. In their study of maximal subgroups of topologically full groups~\cite{GV24}, Grigorchuk and Vorobets mention two desiderata: a finitely generated periodic branch group with a maximal subgroup of infinite index---which was constructed by Garciarena and the author in~\cite{GP24}---and a group exemplifying a positive answer to the following question:

\begin{que}\label{que:main}\cite[Problem~3, attributed to Grigorchuk]{Thi25}
	Does there exist a finitely generated branch group with uncountably many maximal subgroups?
\end{que}

In their recent work~\cite{KS24}, Kionke and Schesler showed that every finitely generated residually finite group $G$ may be embedded into a branch group $\Gamma$ such that certain properties of $G$, such as amenability or periodicity, carry over to $\Gamma$. To address the question posed above, ideas of Kionke and Schesler are combined with methods developed by Garciarena and the author in~\cite{GP24} to establish the following theorem, which yields a group exhibiting the desired properties.

\begin{thm}\label{thm:main}
	Let $G$ be a countably based residually-(finite perfect) group. There exists a layered branch group $\Gamma$ containing a subgroup isomorphic to $G$ and permitting an injective map from the set of non-normal maximal subgroups of $G$ into the set of maximal subgroups of $\Gamma$. If $G$ is finitely generated, so is $\Gamma$.
\end{thm}

For the definition of `layered' and `countably based' see \cref{sub:residually_mathcal_c_groups}; every countable residually-(finite perfect) group is `countably based'. That the result is proven by the explicit construction of maximal subgroups given the maximal subgroups of~$G$.

As a coda, this paper contains some topical problems.


\section{Generalities and branch groups} 
\label{sec:generalities_and_branch_groups}

\subsection*{General notation} 
\label{sub:general_notation}

The natural numbers $\N$ include $0$, while $\N_+ = \N\smallsetminus\{0\}$. The left-shift operator (on any kind of sequence) is denoted $\sigma$. For a group~$G$, the set of non-normal maximal subgroups (that is, maximal subgroups that are not normal, \emph{not} subgroups that are maximal among non-normal subgroups) is denoted~$\Max(G)$. By convention, conjugation and commutator are defined $g^h = h^{-1}gh$ and $[g, h] = g^{-1}g^h$, respectively.


\subsection*{Rooted trees} 
\label{sub:rooted_trees}

Let $\mathfrak{X} = (X_n)_{n\in\N_+}$ be a sequence of finite sets of cardinality at least~$2$. The \emph{rooted tree defined by $\mathfrak{X}$} is the graph $T = T_\mathfrak{X}$ consisting of all strings $x_1\dots x_n$ with $x_i \in X_i$ and $n \in \N$, with edges between strings of the form $x_1 \dots x_n$ and $x_1\dots x_n x_{n+1}$. The empty string $\epsilon$ is called the \emph{root} of the tree. For ease of reading, we write $T_\mathfrak{X}.\sigma$ for $T_{\mathfrak{X}.\sigma}$. For a vertex $v \in T$, write $|v|$ for the distance of $v$ to the root, i.e.\ the length of the string $v$. The \emph{$n$\numth layer} $\mathcal{L}_T(n)$ is the set of all vertices of length $n \in \N$.

A \emph{ray} in $T$ is a sequence $\mathfrak{u} = (u_n)_{n \in \N}$ of vertices such that $u_n \in \mathcal{L}_T(n)$ and $u_n$ and $u_{n+1}$ are adjacent for all $n \in \N$. A \emph{spinal sequence for a ray~$\mathfrak{u}$} is a sequence~$\mathfrak{x} = (x_n)_{n \in \N_+}$ such that $x_{n} \in X_{n}$ and $u_n x_{n+1} \neq u_{n+1}$ for all $n \in \N$.


\subsection*{Automorphisms} 
\label{sub:automorphisms}

The group $\Aut(T)$ of (graph) automorphisms of the rooted tree $T = T_\mathfrak{X}$ decomposes as the wreath product $\Aut(T) = \Aut(T.\sigma)\wr \Sym(X_1)$, allowing for iteration. Given $g \in \Aut(T)$ and $x \in X_1$, we write $g|_x$ for the image of the base part of $g$ in the above wreath product under the projection to the component corresponding to $x$. For any $v = x_1 \dots x_n \in T$, set $g|_{v} = g|_{x_1}|_{x_2}\dots|_{x_n}$. The element~$g|_v \in \Aut(T.\sigma^n)$ is called the \emph{section at $v$}. For a set $D \subseteq \Aut(T)$, put $D|_{(n)} = \{ d|_v \mid d \in D, v \in \mathcal{L}_T(n) \}$. For $g, h \in \Aut(T)$ and $v \in T$ we have
\[
	(gh)|_v = g|_{v}h|_{v.g}.
\]


\subsection*{Contraction} 
\label{sub:contraction}

Two sequences $(a_n)_{n \in \N}$ and $(b_n)_{n \in \N}$ are called \emph{cofinal} if there exists~$n \in \N$ such that $a_N = b_N$ for all $N \geq n$. The equivalence class of $(a_n)_{n \in \N}$ under the corresponding equivalence relation is denoted $\overline{(a_n)_{n \in \N}}^{\mathrm{cf}}$.

Let $\mathfrak{D} = (D_n)_{n \in \N}$ be a sequence of subsets $D_n \subseteq \Aut(T.\sigma^n)$. A group $G \leq \Aut(T)$ is called \emph{contracting with respect to $\overline{\mathfrak{D}}^\mathrm{cf}$} if for every $g \in G$ there exists $n \in \N$ such that $g|_{v} \in D_{|v|}$ for all $v$ with $|v| \geq n$. If $G$ is both contracting with respect to $\overline{\mathfrak{D}}^{\mathrm{cf}}$ and $\overline{\mathfrak{E}}^{\mathrm{cf}}$, it is contracting with respect to $\overline{\mathfrak{D} \cap \mathfrak{E}}^{\mathrm{cf}}$, where the intersection is taken pointwise. Thus there exists a unique inclusion-minimal cofinality class~$\overline{\mathfrak{N}}^{\mathrm{cf}}$, called the \emph{nuclear sequence of $G$}, such that $G$ is contracting with respect to it.\footnote{In the literature, contraction is mostly studied in the case of \emph{self-similar groups} acting on \emph{regular} rooted trees, i.e.\ in the case $T = T.\sigma$ and of a group $G$ satisfying $g|_u \in G$ for all $u \in T$ and $g \in G$. In that setting, the group $G$ is called \emph{contracting} (on its own) if its nuclear sequence may be represented by a constant sequence $(D)_{n \in \N}$, for a finite set $D$.}


\subsection*{Branch groups} 
\label{sub:branch_groups}

A group $G \leq \Aut(T)$ is said to act \emph{spherically transitive} if it acts transitively on every layer of $T$. For every $v \in T$, denote by $\st_G(v)$ the pointwise stabiliser of $v$ and by $\St_G(n)$ the pointwise stabiliser of $\mathcal{L}_T(n)$. The \emph{rigid vertex stabiliser of $v$} is the subgroup $\rist_G(v) \leq \St_G(|v|)$ consisting of all elements $g$ such that $g|_u = \id$ for all $u \in \mathcal{L}_T(|v|)\smallsetminus\{v\}$. The \emph{$n$\textsuperscript{th} rigid layer stabiliser} is the subgroup $\Rist_G(n) = \langle \rist_G(v) \mid v \in \mathcal{L}_T(n)\rangle$. A spherically transitive group $G$ is called a \emph{branch group} if~$\Rist_G(n)$ is a subgroup of finite index for all $n \in \N$.

For $g \in \Aut(T)|_{v}$, write $\ins_v(g)$ for the \emph{insertion of $g$ at $v$}, i.e.~the unique element of~$\rist_{\Aut(T)}(v)$ such that $\ins_v(g)|_v = g$. A group $G$ is called \emph{layered} if
\(
	\rist_{G}(v) = G|_{v}
\)
for all $v \in T$, i.e.~if $\ins_v(g) \in G$ for all $g \in G|_{v}$. Any spherically transitive layered group is a branch group.


\subsection*{Generalised spinal groups} 
\label{sub:generalised_spinal_groups}

A \emph{finitary} element $g$ is one such that there exists some~$n \in \N$ such that $g|_u = \id$ for all $u \in \mathcal{L}(n)$. The minimal such $n$ is called the \emph{depth} of~$g$. If the depth of $g$ is $1$, it is called \emph{rooted}. The finitary elements with labels in a sequence $\mathfrak{S}$ form a subgroup $\Fin_{\mathfrak{S}}(T)$ of $\Aut_{\mathfrak{S}}(T)$. Let $\mathfrak{u} = (u_n)_{n \in \N}$ be a ray in $T$. An element $g$ is called $\mathfrak{u}$-\emph{generalised spinal} if $g|_v$ is finitary for all $v \in T\smallsetminus\{u_n \mid n \in \N\}$. It is called $\mathfrak{u}$-\emph{spinal} if all finitary sections are rooted.
We will repeatedly use the fact that, given a $\mathfrak{u}$-generalised spinal element $g$ and an integer $k \in \N$, the difference $\ins_{u_k}(g|_{u_k}^{-1})g$ is finitary.
A group~$G$ generated by subsets~$A$ and $B$ of finitary elements and $\mathfrak{u}$-generalised spinal elements, respectively, is called a \emph{generalised spinal group}. If $A$ and $B$ consist of rooted and $\mathfrak{u}$-spinal elements only, respectively, the group is called a \emph{spinal group}.
Spinal groups were introduced by Bartholdi, Grigorchuk and {\v S}uni{\'k} in~\cite{BGS03}, and encompass important examples like Grigorchuk's groups and the Gupta--Sidki $p$-groups.

\begin{lemma}\label{lem:contraction}
	Let $G$ be a generalised spinal group generated by the sets $A$ and $B$ of finitary and $\mathfrak{u}$-generalised spinal elements, respectively. The nuclear sequence of $G$ is given by $\overline{(B|_{(n)})_{n \in \N}}^\mathrm{cf}$.
\end{lemma}

\begin{proof}
	The proof uses standard methods, compare e.g.~\cite[Lemma~2.4]{BGS03} or~\cite[Lemma~2.5]{Pet23}. Let $b \in B$ and let $a \in A$ with depth $n$. Then for $v \in \mathcal{L}_T(n)$
	\[
		b^a|_v = a^{-1}|_{v} b|_{v.a^{-1}} a|_{v.b^{-1}a} = b|_{v.a^{-1}} \in B|_{(n)}.
	\]
	Let $g \in G$. Since $G$ is a quotient of the free product $A \ast B$, there exist $a_1, \dots, a_m, a_{m+1} \in A$ and $b_1, \dots, b_m \in B$ such that
	\(
		g = b_1^{a_1} \dots b_m^{a_m} a_{m+1}.
	\)
	Among all such products, choose the one with the minimal number~$m$ and set $\ell_G(g) = m$. Write $n$ for the maximal depth of the finitary elements $a_i$. Let $v \in \mathcal{L}(n)$. Then
	\(
		g|_v = b_1^{a_1}|_v \dots b_m^{a_m}|_v.
	\)
	By the analysis of their sections above, the elements of the form $b_i^{a_i}|_v$ are either finitary or in $B|_{u_n}$. If two adjacent elements are contained in $B|_{u_n}$, also their product is, whence $g|_v$ permits a product of the form given above of length at most $(\ell_G(g)+1)/2$. Repeating this process, it is apparent that there exists some $k \in \N$ such that for all $v \in \mathcal{L}_T(k)$ there exist $\tilde{a}_1, \tilde{a}_2 \in A|_{(k)}$ and $\tilde{b}_1 \in B|_{u_k}$ such that $g|_v = \tilde{b}_1^{\tilde{a}_1} \tilde{a}_2$. Examining the sections at the maximum of the depths of $\tilde{a}_1$ and $\tilde{a}_2$ shows that $G$ is contracting with respect to $\overline{(B|_{(n)})_{n \in \N}}^\mathrm{cf}$. Evaluating the sections of~$B$ shows that $\overline{(B|_{(n)})_{n \in \N}}^\mathrm{cf}$ is minimal.
\end{proof}

Let $\mathfrak{u}$ be a ray in $T$ and let $\mathfrak{x}$ be a spinal sequence for $\mathfrak{u}$. Denote by $\operatorname{Sp}(\mathfrak{u}, \mathfrak{x})$ the group of all $\mathfrak{u}$-spinal elements such that $g|_{u_ny} = \id$ for all $n \in \N$ and $y \in X_{n+1}$ such that $y \neq x_{n+1}$ and $u_ny \neq u_{n+1}$, compare to the `G groups' of \cite{BGS03}. The group $\operatorname{Sp}(\mathfrak{u}, \mathfrak{x})$ is isomorphic to the direct product $\prod_{n \in \N_+} \Sym(X_n)$.


\subsection*{Residually-$\mathcal{C}$ groups} 
\label{sub:residually_mathcal_c_groups}

Let $\mathcal{C}$ be a class of groups. A \emph{$\mathcal{C}$-residual representation} of a group $G$ is a homomorphism $\varphi \colon G \to \prod_{i \in I} S_i$ such that $S_i \in \mathcal{C}$ for all $i \in I$ and~$G.\varphi$ is a subdirect product. A $\mathcal{C}$-residual representation is called \emph{countably based} if the index set $I$ is countable; in this case, we assume $I = \N_+$ without loss of generality. It is furthermore called \emph{infinitary} if $G.\varphi \cap \bigoplus_{n \in \N_+} S_i = 1$, or, equivalently, if all powers of the homomorphism induced by the left-shift $\sigma$ on the sequence $(S_n)_{n \in \N_+}$ have trivial kernel. A group is called \emph{(countably based) residually-$\mathcal{C}$} if it admits a faithful (countably based) $\mathcal{C}$-residual representation. Every countable residually-$\mathcal{C}$ group is countably based residually-$\mathcal{C}$.

\begin{lemma}
	Every countably based residually-$\mathcal{C}$ group admits a faithful infinitary $\mathcal{C}$-residual representation.
\end{lemma}

\begin{proof}
	Without loss of generality, let $G \leq \prod_{n \in \N_+} S_n$ for some non-trivial groups $S_n \in \mathcal{C}$. Denote by $\delta_n \colon S_n \to \prod_{m \in \N_+} S_n$ the diagonal embedding $g \mapsto (g, g, \dots)$ and consider the induced injection $\delta \colon \prod_{n \in \N_+} S_n \to \prod_{(n,m) \in \N_+^2} S_n$ on the product. The image of $G$ under $\delta$ is a subdirect product of $\prod_{(n,m) \in \N_+^2} S_n$ such that removing any finite number of components does not change the group up to isomorphism.
\end{proof}

In the following, the class of finite perfect groups will be used for $\mathcal{C}$. The following classical result is of need.

\begin{theorem}[Gaschütz--Itô]\label{thm:gaschuetz-ito}
	For every finite insoluble group $P$ there exists a finite set $\tau(P)$ such that $P$ admits a faithful non-regular transitive action on $\tau(P)$.
\end{theorem}

This statement is a reformulation of Gaschütz and Itô's original theorem, see~\cite[Satz~5.7]{Hup67}, which says that groups with all minimal subgroups normal are soluble. For every non-trivial perfect group, the cardinality of $\tau(P)$ is at least~$5$.



\section{Proof of the Main theorem} 
\label{sec:proof_of_the_main_theorem}

Fix a countably based residually-(finite perfect) group $G$ with a faithful infinitary residually-(finite perfect) representation $\varphi \colon G \to \prod_{n \in \N_+} S_n$, fix a non-trivial finite perfect group $S_0$, and put $\mathfrak{S} = (S_n)_{n \in \N}$. Put $\mathfrak{X} = (\tau(S_{n-1}))_{n \in \N_+}$, using a fixed assignment~$\tau$ as in \cref{thm:gaschuetz-ito}. Fix a ray $\mathfrak{u}$ in $T = T_{\mathfrak{X}}$ and fix a spinal sequence~$\mathfrak{x}$ for~$\mathfrak{u}$ such that $\st_{S_n}(x_{n+1}) \neq \st_{S_n}(y_{n+1})$ for $u_ny_{n+1} = u_{n+1}$ and for all $n \in \N$. Such a sequence exists, since $S_n$ acts faithfully and non-regular on $\tau(S_{n})$. The group $\prod_{n \in \N_+} S_n$ naturally embeds into $\prod_{n \in \N_+} \Sym(\tau(S_n))$, which is isomorphic to $\operatorname{Sp}(\mathfrak{u}, \mathfrak{x})$. For convenience, the group $G$ is identified with its isomorphic image in $\operatorname{Sp}(\mathfrak{u}, \mathfrak{x})$, i.e.\ every $g \in G$ is identified with the $\mathfrak{u}$-spinal automorphism $s_g$ given by
\[
	s_g|_{u_n x_{n+1}} = g.\pi_{n+1} \quad\text{and}\quad s_g|_{u_ny} = \id
\]
for all $n \in \N$ and $y \in \tau(S_{n+1})$ such that $u_ny \notin \{u_nx_{n+1}, u_{n+1}\}$, where $\pi_n$ denotes the projection to the $n$\textsuperscript{th} component of the direct product $\prod_{n \in \N_+} S_n$, viewed as a rooted automorphism. Since $G$ is infinitary represented, the left-shift $\sigma$ induces an isomorphism of $G$ onto its image.

Consider the spinal group
\[
	\Gamma = \langle S_0 \cup G \rangle,
\]
where $S_0$ acts by rooted automorphisms. Clearly $\Gamma$ is finitely generated if $G$ is. In the remainder, all of the above is fixed. By construction, the group $\Gamma$ is contained in $\Aut_\mathfrak{S}(T)$. For every $m \in \N$, the groups $G.\sigma^m$ permit the identity map as infinitary residually-(finite perfect) representation within $\prod_{n \in \N_+} S_{n + m}$. Choose the shifts $\mathfrak{X}.\sigma^m, \mathfrak{u}.\sigma^m$ and $\mathfrak{x}.\sigma^m$ to embed $G.\sigma^m$ into $\Aut(T.\sigma^m)$, and put $\Gamma.\sigma^m = \langle S_m \cup G.\sigma^m \rangle$.

The construction of $\Gamma$ highly depends on the assignment $\tau$ and the representation~$\varphi$. On the other hand, the choice of $\mathfrak{u}$ and $\mathfrak{x}$ is less substantial and corresponds to conjugation within $\Aut(T)$.

\begin{proposition}\label{prop:Gamma branch}
	The group $\Gamma$ is a layered branch group.
\end{proposition}

\begin{proof}
	Note first that $G|_{u_1} = G.\sigma$ and $g = \ins_{x_1}(g|_{x_1})\ins_{u_1}(g|_{u_1})$ for all $g \in G$. Since~$S_0$ acts transitively on $X_1$ and the projection of $G$ onto $S_1$ is surjective, $\St_{\Gamma}(1)|_{y} \geq \langle S_1 \cup G.\sigma \rangle$ for all $y \in X_1$.
	At the same time, $\Gamma|_y$ is contained in $\langle \bigcup_{z \in X_1} S_0|_z \cup G|_z \rangle = \Gamma.\sigma$, hence $\St_{\Gamma}(1)|_{y} = \Gamma.\sigma$. Arguing for $\Gamma.\sigma$ just as for $\Gamma$ (and so on), one finds that $\St_\Gamma(n)|_{y}$ acts transitively on $X_{n+1}$ for all $n \in \N$ and $y \in \mathcal{L}_T(n)$, from which it directly follows that $\Gamma$ acts spherically transitive.

	Let $s \in \st_{S_0}(x_1)\smallsetminus \st_{S_0}(u_1)$, such an element exists by the choice of $\mathfrak{x}$ and since $S_0$ acts transitively on $\tau(S_0)$. Let $g, h \in G$. Then $[g, h^s]|_y = [g|_y, h^s|_y] = [g|_y, h|_{y.s^{-1}}]$. If $y \notin \{u_1, x_1\}$, the section $g|_y$ is trivial (and so is $[g, h^s]|_y$). But $u_1.s^{-1} \notin \{u_1, x_1\}$ either, whence $[g, h^s]|_{u_1} = \id$. Thus $[g, h^s] = \ins_{x_1}([g|_{x_1}, h|_{x_1}]) \in \rist_\Gamma(x_1)$. Since $G$ is a subdirect product, $G|_{x_1} = S_1$. Since $S_1$ is perfect, $\rist_\Gamma(x_1) \geq S_1$. By conjugation with $S_0$ (acting transitively), $\rist_\Gamma(x) \geq S_1$ for all $x \in X_1$. For every $g \in G$ one computes $\ins_{x_1}(g|_{x_1}^{-1})g = \ins_{u_1}(g|_{u_1}) \in \Gamma$. Thus $\rist_\Gamma(x) \geq \langle G|_{u_1} \cup S_1 \rangle = \Gamma.\sigma$. The result follows by arguing in the same way for the shifts $\Gamma.\sigma^n$.
\end{proof}

By the above, $\Fin_{\mathfrak{S}}(T) \leq \Gamma$. Define a map $\Delta$ from the set of subgroups of $G$ to the set of subgroups of $\Gamma$ as follows; for any subgroup $H \leq G$ set
\[
	\Delta(H) = \langle \Fin_\mathfrak{S}(T) \cup H \rangle.
\]
By definition, $\Delta(H)$ is a generalised spinal group for every subgroup $H \leq G$.

\begin{lemma}\label{lem:delta(M) layered}
	The group $\Delta(H)$ is a layered branch group for every $H \leq G$.
\end{lemma}

\begin{proof}
	Since $\Fin_\mathfrak{S}(T)$ acts spherically transitive, so does $\Delta(H)$. Let $m \in M$ and let~$n \in \N$. Since $\Fin_\mathfrak{S}(T)$ is layered, it is sufficient to show $\ins_{u_n}(h|_v) \in \rist_{\Delta(H)}(u_n)$ for all $v \in \mathcal{L}_T(n)$. Since $h$ is $\mathfrak{u}$-spinal, $\ins_{u_n}(h|_{u_n}^{-1})h \in \Fin_{\mathfrak{S}}(T)$, and $\ins_{u_n}(h|_{u_n}) \in \Delta(H)$. The elements $\ins_{u_n}(h|_v)$ are finitary, whence also contained in $\Delta(H)$.
\end{proof}

\begin{lemma}\label{lem:delta(M) proper}
	The subgroup $\Delta(H) \leq \Gamma$ is proper for every proper subgroup $H < G$.
\end{lemma}

\begin{proof}
	Both $\Delta(H)$ and $\Gamma$ are $\mathfrak{u}$-generalised spinal groups. Thus, by \cref{lem:contraction}, they are contracting with nuclear sequence $\overline{(H.\sigma^n)_{n \in \N}}^\mathrm{cf}$ and $\overline{(G.\sigma^n)_{n \in \N}}^\mathrm{cf}$, respectively. Since the shifts $\sigma^n \colon G \to G.\sigma^n$ are bijections, the inclusions~$H.\sigma^n < G.\sigma^n$ are proper for all $n \in \N$, whence the nuclear sequences of $\Delta(H)$ and $G$ are not cofinal. Therefore $\Delta(H)$ is a proper subgroup.
\end{proof}

\begin{lemma}\label{lem:delta injective}
	The map $\Delta$ is injective on the set of maximal subgroups of $G$.
\end{lemma}

\begin{proof}
	Let $M$ and $\tilde{M}$ be two maximal subgroups of $G$ and assume that $\Delta(M) = \Delta(\tilde{M})$. Then $M \leq \Delta(\tilde{M})$, hence $G = \langle M \cup \tilde{M} \rangle \leq \Delta(\tilde{M})$. Clearly $S_0 \leq \Fin_{\mathfrak{S}}(T)$, thus $\Gamma = \langle S_0 \cup G \rangle = \Delta(\tilde{M})$, which contradicts \cref{lem:delta(M) proper}.
\end{proof}

\begin{proposition}\label{prop:delta(M) maximal}
	The subgroup $\Delta(M) \leq \Gamma$ is maximal for every $M \in \Max(G)$.
\end{proposition}

\begin{proof}	
	Let $g \in \Gamma \smallsetminus \Delta(M)$, put $\Lambda = \langle \Delta(M) \cup \{g\}\rangle$, and let $k \in \N$ be such that $g|_v \in G|_{(k)}$ for all $v \in \mathcal{L}_T(k)$; this number exists by \cref{lem:contraction}. Since $g$ is not finitary, there is a section~$g|_v$ of $g$ contained in $G.\sigma^k$. Assume that all non-finitary sections~$g|_v$ are contained in~$M.\sigma^k$. By \cref{lem:delta(M) layered}, ~$\Delta(M)$ is layered, hence the element $h := \prod_{v \in \mathcal{L}_T(k)} \ins_v(g|_v)$ is contained in $\St_{\Delta(M)}(k)$, whence $gh^{-1}$ is finitary and $g \in \Delta(M)$, a contradiction. Thus there is a section $g|_v \in (G \smallsetminus M).\sigma^k$. Since $\Fin_\mathfrak{S}(T)$ acts spherically transitive, without loss of generality $v = u_k$, replacing $g$ with a $\Fin_\mathfrak{S}(T)$-conjugate if necessary. Similarly, it may be assumed that $g \in \st(u_k)$.
	Since $M \not\trianglelefteq G$, we find $\Norm_{G}(M) = M$; so let $m \in M$ such that $m^{g} \notin M$. Since $\Delta(M)$ is layered, $\ins_{u_k}(m|_{u_k})^g = \ins_{u_k}(m^g|_{u_k}) \in \Lambda$. But $\ins_{u_k}(m^g|_{u_k}^{-1})m \in \Fin_\mathfrak{S}(T)$, hence $m^g \in \Lambda$. The maximality of $M$ now ensures $G = \langle \{m^g\} \cup M \rangle \leq \Lambda$, whence $\Gamma = \langle S_0 \cup G \rangle \leq \Lambda \leq \Gamma$. Therefore, $\Delta(M)$ is maximal.
\end{proof}

The \hyperref[thm:main]{Main Theorem} is a consequence of \cref{prop:Gamma branch}, \cref{lem:delta injective} and \cref{prop:delta(M) maximal}. The following corollary, which provides a positive answer to the \hyperref[que:main]{Main Question}, is obtained by applying it to a finitely generated countably based residually-(finite perfect) group with uncountably many maximal subgroups (note that a finitely generated group permits at most countably many normal maximal subgroups); for example to a non-cyclic free group of finite rank, which are residually alternating by a result of Magnus and Katz in~\cite{KM69}.

\begin{corollary}\label{cor:main}
	There exists a finitely generated branch group with uncountably many maximal subgroups.
\end{corollary}

\begin{remark}
	A statement analogous to the \hyperref[thm:main]{Main Theorem} can be established for a countably afforded residually finite group $H$, say $H \leq \prod_{n\in\N} T_n$. Let~$\Gamma$ be the group constructed as above, for a countably based residually-(perfect alternating) group (e.g.\ the free group of rank two). Every group $T_n$ embeds into some sufficiently large alternating group. Pick a spinal sequence with no term equal to the spinal sequence employed for the construction of~$\Gamma$. Using similar methods as for the \hyperref[thm:main]{Main Theorem}, it can be proved that the group $\langle \Gamma \cup H \rangle$ is layered branch and contains maximal subgroups $\langle \Gamma \cup M \rangle$ for all $M \in \Max(H)$, since the spinal elements in $G$ and $H$ can be distinguished by their spinal sequences.	
\end{remark}

\begin{remark}
	The construction in the previous remark furthermore establishes that every countably based residually finite group embeds into a layered branch group, which may serve as a minor adjunct to the main result of \cite{KS24}.
\end{remark}


\section{Questions and directions} 
\label{sec:questions_and_directions}

This section contains a number of problems pertaining to the \hyperref[que:main]{Main Question}. A rooted tree is called \emph{$m$-regular} if all vertices, excluding the root, possess valency $m+1$, while the root’s valency is $m$. Although branch groups acting on trees of increasing valency have recently received increased attention, as evidenced by~\cite{Pet24, ST24}, branch groups acting on $m$-regular rooted trees hold a particular significance.

\begin{question}
	Does there exist a finitely generated branch group acting on a regular rooted tree that admits uncountably many maximal subgroups?
\end{question}

The construction presented in the preceding section cannot provide a positive answer. If a finitely generated group~$G$ is afforded by a sequence of groups of bounded order, it is necessarily finite by the resolution of the restricted Burnside problem.

\begin{question}
	Is the map $\Delta$ a bijection between $\Max(G)$ and the set of maximal subgroups of infinite index in $\Gamma$?
\end{question}

In \cite{Fra20}, Francoeur demonstrated that every maximal subgroup of infinite index in a branch group is itself a branch group. Considering the groups $\Gamma$ and their maximal subgroups $\Delta(M)$, the following question is natural:

\begin{question}
	Let $G$ be a finitely generated branch group. Is every infinite index branch subgroup of $G$ contained in a (branch) infinite index maximal subgroup of $G$?
\end{question}



\providecommand{\bysame}{\leavevmode\hbox to3em{\hrulefill}\thinspace}
\providecommand{\MR}{\relax\ifhmode\unskip\space\fi MR }
\providecommand{\MRhref}[2]{%
  \href{http://www.ams.org/mathscinet-getitem?mr=#1}{#2}
}
\providecommand{\href}[2]{#2}

\end{document}